\documentclass[11pt]{amsart}
\usepackage{amsmath}
\usepackage{amssymb}
\usepackage{amsfonts}
\usepackage{amsthm}
\usepackage{verbatim}
\usepackage{amscd}
\usepackage{cite}
\usepackage{leftidx}
\usepackage{enumerate}
\usepackage{txfonts}
\usepackage{manfnt}
\usepackage{amscd}
\usepackage[mathscr]{eucal}
\usepackage{hyperref}
\usepackage{mathpazo}
\def\cP{\mathcal P}
\def\cX{\mathcal X}
\def\tY{\tilde Y}

\newtheorem{thm}{Theorem} % [section]

\newtheorem*{thm*}{Theorem}
\newtheorem*{prop*}{Proposition}
\newtheorem{cor}[thm]{Corollary}
\newtheorem*{cor*}{Corollary}

\newtheorem{lem}[thm]{Lemma}
\newtheorem*{lem*}{Lemma}

\newtheorem*{claim*}{Claim}
\newtheorem{prop}[thm]{Proposition}

\newtheorem{defn}[thm]{Definition}%number together with conj
\theoremstyle{remark}

\newtheorem{rem}[thm]{Remark}
\newtheorem*{rem*}{Remark}
\newtheorem{crit-rem}[thm]{Critical remark}
\newtheorem{remarks}[thm]{Remarks}

\newtheorem{example}[thm]{Example}
\newtheorem*{example*}{Example}

\newtheorem*{defn*}{Definition}

\def\inv{^{-1}}

\def\refp #1.{(\ref{#1})}
\newcommand\carets [1]{\langle #1 \rangle}

\newcommand{\A}{\mathcal{A}}

\newcommand{\M}{\mathcal{M}}

\newcommand{\cM}{\mathcal M}

\newcommand{\Cal}[1]{\mathcal #1}

\newcommand{\ul}[1]{\underline {#1}}

\def\sbr #1.{^{[#1]}}
\def\sfl #1.{^{\lfloor #1\rfloor}}

\def\inv{^{-1}}
\def\?{{\bf{??}}}

\def\H{\mathcal H}

\def\M{\Cal M}
\def\A{\Bbb A}

\def\C{\mathbb C}
\def\P{\mathbb P}

\def\sym{\text{\rm Sym} }

\def\O{\mathcal O}

\def\rk{\text{rk}}

\def\g{\mathfrak g}

\def\1/2{\frac{1}{2}}

\def\Sec{\textrm{Sec}}

\def\2{{[2]}}

\def\nl{\newline}

\def\<{\langle}
\def\>{\rangle}

\def\2{{[2]}}

\def\scl #1.{^{\lceil#1\rceil}}
\def\spr #1.{^{(#1)}}
\def\sbc #1.{^{\{#1\}}}

\def\subpr#1.{_{(#1)}}

\def\beq{\begin{equation*}}
\def\eeq{\end{equation*}}

\def\g3{{\Gamma\spr 3.}}

\newcommand{\eqspl}[2]{
\begin{equation}\label{#1}
\begin{split}
#2\end{split}\end{equation}}

\newcommand{\exseq}[3]{
0\to #1\to #2\to #3\to 0
}

\newcommand{\beginalphaenum}{
\begin{enumerate}\renewcommand{\labelenumi}{ }
\item \begin{enumerate}
}

\def\eex{\end{rm}\end{example}}
\begin{document} 

%\end{document}
%\title{Differential complexes, degeneracy\\ and Hodge theory on Poisson manifolds}
\title{Low-degree Rational curves on hypersurfaces\\
 in projective spaces  and their fan degenerations}
\author %{author}
{Ziv Ran}
%\end{document}
%\Large

%\thanks{\raggedright{

%Partially supported by NSA Grant MDA904-02-1-0094} }
\thanks{arxiv.org 1906.03747}
\date {\today}% \enddate

%\affil University of California, Riverside\endaffil

\address {\nl UC Math Dept. \nl
Skye Surge Facility, Aberdeen-Inverness Road
\nl
Riverside CA 92521 US\nl 
ziv.ran @  ucr.edu\nl
\url{http://math.ucr.edu/~ziv/}
}

%\email {ziv.ran @ucr.edu}
 \subjclass[2010]{14n25, 14j45, 14m22}

\begin{abstract}
	We study rational curves on general Fano hypersurfaces
	in projective space, mostly by degenerating
	the hypersurface along with its ambient projective space to reducible varieties.
We prove results on existence of low-degree rational curves with balanced normal bundle, 
 and reprove some results on irreducibility of spaces of rational curves of low degree.
\end{abstract}
\maketitle
\tableofcontents
\section*{Introduction}
The celebrated work of Mori starting with \cite{mori-ample-tangent} had brought out the importance of rational curves
in higher-dimensional geometry, and
in the early 1990s Campana %\cite{campana} 
and Koll\'ar-Miyaoka-Mori % \cite{kmm} 
applied
Mori's bend-and-break method to prove that the family of rational curves
on any Fano manifold $X$ is a large enough  to connect a pair of general points
(i.e. $X$ is 'rationally connected'). 
This work and much more is exposed in Koll\'ar's book  \cite{kollar-rat-curves}.
 Since then there has been considerable interest, especially
by Joe Harris and his school (e.g. \cite{harris-roth-starr}), in studying the family of rational curves on a general
Fano hypersurface in projective space, especially as to dimension and irreducibility.
The expected dimension of the family of rational curves of degree $e$ in a hypersurface
$X$ of degree $d$ in $\P^n$ is $e(n-d+1)+n-4$, and it has
been conjectured that when $X$ is general, this family is irreducible and of the expected dimension
provided $d\leq n-1$ and $n>3$. This conjecture has  been proven for $d\leq n-2$ by Riedl and Yang
\cite{riedl-yang} which also contains extensive references. See also \cite{tseng-note} for a partial extension
to the case $d=n-1$. \par
The purpose of this paper is to bring to bear on these questions a technique used previously
\cite{sing},\cite{enriques} to study curves (of any geometric genus) in the projective plane.
This technique is based on degenerating the ambient projective space itself to a reducible
variety called a \emph{fan} and simultaneously degenerate a general hypersurface in projective
space to a general hypersurface of suitable type on the fan. In a sense, one is degenerating the entire
family of projective hypersurfaces of given degree to an analogous family- of the same dimension-
on a fan. This has the advantage that the components of the limit are simpler, consisting,
 e.g. in the case of a 2-component fan, of a  hypersurface of degree
 $d-1$ and a blowup of a (rational) hypersurface of degree $d$ with a point of  multiplicity $d-1$,
called a quasi-cone. Moreover the union has double points only.
We call this a 2-fan of type $(d, d_2)$. We will focus mainly on the case $d_2=d-1$, where
$X_1$ is a blowup of a special rational hypersurface called a quasi-cone and is in turn the blowup of $\P^{n-1}$
in a $(d, d-1)$ complete intersection.
\par
Now the limit on a 2-component fan $X_1\cup X_2$ of a rational curve of degree $e$ on a general
hypersurface of degree $d$ essentially takes the form $C_1\cup C_2$ 
where $C_1$ is a birational transform of a rational curve of degree $e_1$ in $\P^{n-1}$ that is
$(e_1(d-1)-e_2)$-secant to a certain $(d, d-1)$ complete intersection, while $C_2$ is a rational
curve of degree $e_2\geq 0$ on $X_2$, a general hypersurface of degree $d-1$ and $e_1+e_2=e$. 
This leads us to study
secant rational curves to complete intersections. 
We will show in Thm  \ref{secant-thm} that for fixed $d, e, a$ in a suitable range,  the locus of rational curves
of degree $e$ that are $a$-secant to a general $(d, d-1)$ complete intersection is irreducible
and of the expected dimension. This will be used to reprove 
 a result (Thm \ref{main}) on irreducibility of families
of rational curves  of low degree on general hypersurfaces 
of degree $d<n$ in $\P^n$. 
This result is mostly subsumed by the theorem of Riedl-Yang and, as noted
by the referee, can also be deduced from the results of Harris-Roth-Starr
\cite{harris-roth-starr}. Those proofs are different.
It is interesting
to note that for the bend-and-break method as used in \cite{riedl-yang} and elsewhere, 
the case of low-degree curves is hardest, but for our method the opposite is true.
Also, in \S \ref{through-pt}, we prove 
some results on irreduciblity of the family of rational
curves going through a fixed, general point.
\par
Our main new result (Theorem \ref{balanced-high})
concerns the existence of rational curves of low 
degree with balanced normal bundle, 
and extends
a result of Coskun-Riedl \cite{coskun-riedl}:  
\begin{thm*}
	For $d\leq n, n\geq 4, e\leq 2n-2$ there exists on a general hypersurface
	of degree $d$ in $\P^n$ a rational curve of degree $e$ with 
	balanced normal bundle.
	\end{thm*}
The detailed contents of the paper are as follows.\begin{itemize}\item 
We begin in \S1 by studying
certain (decomposable) bundles on rational chains and combs which smooth out to
balanced bundles on rational curves.\item
 This study is applied in \S2 to 
polygons on a union of hyperplanes to give a preliminary construction
for rational curves of low degree with balanced normal bundles.
\item In \S3 we study certain rational
projective hypersurfaces called quasi cones which arise,
essentially as componenets, in the study
of hypersurfaces on fans. \item
Fans  themselves
 along with their deformations and hypersurfaces on them
 are introduced in \S 4. \item In \S5 we prove a technical result concerning what happens
to a curve on a general hypersurface in projective space as the latter
degenerates to a fan. 
\item In \S6 we study families of
curves in a given homology
class on a quasi-cone. Using the identification of the quasi-cone
with a blowup of projective space, these can be identified with families of curves
in projective space that are multisecant to a given codimension-2
complete intersection.\item The results on balanced normal bundles 
and irreducibility of families are then proved in \S7 and \S8
using fans.\end{itemize}
Throughout the paper, we work over $\C$.
\vskip.35cm
I thank the referee for many helpful comments on both the micro and macro levels.
\section{Balanced bundles on chains and combs}

Recall that a  vector bundle $E$ on a rational curve  $C$ is said to be balanced
if
\[E\simeq r_+\O_C(a+1)\oplus (r-r_+)\O_C(a), r_+>0\] 
where $\O_C(*)$ denotes the unique line bundle
of degree $*$.
Balancedness is an open property on vector bundles. The subbundle $E_+=r_+\O_C(a+1)$ is uniquely
determined and called the \emph{upper} subbundle. 
We can write $E_+=V_+(E)\otimes\O_C(a+1)$ where the vector space $V_+(E)$
is canonical up to scalar multiplication and for any point $p\in C$,
$V_+(E)$ may be identified, again canonically up to scalar,
with a uniquely determined subspace of the  fibre $E(p)$, called the
\emph{upper subspace} at $p$.  \par
Now let \[C=C_1\cup...\cup C_e\] 
be a (connected, nodal) tree of smooth
rational curves. 
We denote by  $\O_C(a_1,..., a_e)$ any line bundle
on $C$ whose restriction on each $C_j$ is $\O_{C_j}(a_j)$.
A  \emph{balanced decopmosition} for a vector bundle $E$ on $C$
is an isomorphism to a direct sum of line bundles
whose total degrees differ by at most 1, i.e.
\eqspl{upper}
{E\simeq \bigoplus\limits_{i=1}^{r_+}\O_{C}(a_{i1},..., a_{ie})
\oplus  \bigoplus\limits_{i=1}^{r-r_+}\O_{C}(b_{i1},...,b_{ie})
}
where \[\sum\limits_{j=1}^e a_{ij}=:\deg_+(E), i=1,...,r_+,\sum\limits_{j=1}^e b_{ij}=\deg_+(E)-1, i=1,...,r-r_+.\]
The first sum above is called an upper subbundle of $E$.
 Thus $r(\deg_+(E)-1)+r_+=\deg(E).$
Then $\deg_+(E)$ is called the upper degree and $r_+=r_+(E)$ is called the upper rank.
A vector bundle $E$ on a rational tree $C$ is said to be \emph{balanced}
if its restriction on any connected subtree 
admits a balanced decomposition.
A bundle $E$ on a rational tree
 is said to be \emph{strongly balanced} if it is balanced and there is a subbundle
$E_+\subset E$- necessarily unique-
 that restricts on each component $C_i$ to the upper subbundle
of $E_i$.\par 
\begin{remarks}

	\par
(i)	Clearly every strongly balanced bundle is balanced
with upper degree and rank being the degree and rank of $E_+$. 
Indeed strongly balanced bundles behave just like balanced bundles on $\P^1$ and
in particular there is a well-defined upper subspace $V_+(E)$.
\par (ii) Unlike in the irreducible case, 
not every balanced bundle is strongly balanced
because an 'upper' subbundle is in general not unique when
some of the $a_{ij}$ may not equal the upper degree of $E_j$. \par
(iii) Another subtlety of weakly (non-strongly) balanced bundles
is that 'upper' line subbundles may be rigid and it is not in
general possible to form linear combinations of them.\par
(iv) The condition that just $E$ itself,
rather than all its restrictions, admit a{balanced} decomposition
seems too weak. In particular Corollary \ref{smoothing} 
below seems unlikely to hold in that
generality.\par
(v) Balancedness is defined only for vector bundles which split
as direct sums of line bundles and it is \emph{not} claimed that any vector bundle on a rational chain splits thusly.
\end{remarks}
We now consider a connected chain of rational trees
\[C=C_1\cup...\cup C_e\]
where each link $C_i$
is itself a rational tree (rather than a $\P^1$) and where
each $E_i=E|_{C_i}$ is strongly balanced. Set $p_i=C_i\cap C_{i+1}$.
Note that if each $E_i$ is strongly balanced then $V_+(E_i)$, 
being a subspace of both $E_i(p_{i-1})$ and $E_i(p_i)$,  may be 
intersected with both $V_+(E_{i-1})$ and $V_+(E_{i+1})$.
Repeating this operation we obtain a pair of descending flags, 'forward' and 'backward' on $V_+(E_i)$.
We will say that $E$ is in \emph{ balanced general position} if
 these flags each have the generic dimension (e.g. the first member of the forward flag
 is of dimension $\max(r_+(E_i)+r_+(E_{i+1})-r, 0)$ etc.) and the pair is in mutual
general position, $\forall i$.
\begin{lem}\label{balanced-lem}
	Let $C$ be a rational tree of the form $C=C_1\cup...\cup C_e$
where each $C_i$ is itself a rational tree.
Let $E=E_1\cup...\cup E_e$ be a vector bundle $C$ such that each restriction
 $E_i=E|_{C_i}$ 
 is strongly balanced and such that $E$ is in balanced  general position. 
 Then $E$ is balanced. 
\end{lem}

\begin{proof} 
	It suffices to prove $E$ itself 
	admits a balanced decomposition as the hypotheses descend to subchains.
We use induction on $e$ and on the rank $r$. The assertion is trivial if $r=1$,
 vacuous if $e=1$ and easy if $e=2$.
In fact if $e=2$ and $r_+(E_1)+r_+(E_2)>r$, then $E$ 
is clearly strongly balanced.
%has a unique upper subbundle
%of the form $(r_+(E_1)+r_+(E_2)-r)(\O_{C_1}(\deg_+(E_1))\cup\O_{C_2}(\deg_+(E_2))$
%and this case behaves identically to the 1-component case.
\par
Now we will assume $e\geq 3$ and use induction
on $e$. By the last remark, we may assume
\eqspl{rank}
{r_+(E_i)+r_+(E_{i+1})\leq r, \forall i,}
 or else $E|_{C_i\cup C_{i+1}}$ is strongly balanced and
 $C_i\cup C_{i+1}$ may be treated as a single link of the chain.
% in other words we can write $C=C_1\cup...\cup(C_i\cup C_{i+1})\cup...C_e$
% and proceed with the argument below treating $C_i\cup C_{i+1}$
% as a single chain part.
 This together with general position imply that the upper subspaces of $E_i, E_{i+1}$
 at $p_i=C_i\cap C_{i+1}$ have trivial intersection, $\forall i$. 
We switch notation, writing $C$ in the form $C^*_1\cup C_2\cup C_3$ where $C^*_1$
is itself a chain on $e-2$ links. By induction, we may assume
the assertion holds for $C^*_1$ and $C^*_1\cup C_2$,  i.e. $E^*_1:=E_{C^*_1}$
and $E_{12}=E_{C^*_1\cup C_2}$ are both admit balanced decompositions.
We will say that a line subbundle
of $E_{12}$ is of type $(11)$ if its restrictions on $C_1^*$ and $C_2$ are both of the upper
degree, and likewise for $(10)$,
$(111)$ etc. Assume to begin with that
$r_+(E^*_1)+r_+(E_2)>r$.  Then by induction, $E_{12}$
is a sum of line subbundles of types $(11), (01)$ and $(10) $,
where the first two make up the upper subspace at $p_2$. By \eqref{rank} (together with
general position), the first two types may be glued to a direct summand line subbundle
of type $0$ on $E_3$, while the last type may be glued to summands of type 0 or 1.
In total this gives a decomposition of $E$ with summands of types $(110), (010), (101), (100)$, which is
a balanced decomposition.\par
Now suppose $r_+(E^*_1)+r_+(E_2)\leq r$. Then $E_{12}$ is a sum of bundles of types
$(01), (10), (00)$. Again by \eqref{rank}, the first type glues to a summand of type 0 on $E_3$
to yield a summand of type $(010)$. At this point we could use induction on $r$ to conclude.
Alternatively,  suppose now $r_+(E_1^*)+r_+(E_3)>r-r_+(E_2)$. Then summands of type 0 on $E_2$
can be extended to types $(101, (001), (100)$, so we get a balanced decomposition for $E$.
Finally suppose $r_+(E_1^*)+r_+(E_3)\leq r-r_+(E_2)$. Then summands of type 0 on $E_2$
can be extended to types $(100), (001), (000)$, so again we get a balanced decomposition.

\end{proof}
\begin{rem}
	The Lemma says in particular that $E$ is a direct sum of line bundles. It does
	\emph{not} say $E$ is strongly balanced. Also, the Lemma will mainly
	be used in the case where all the $C_i$ are irreducible, but the more
	general statement is convenient for the induction.
	\end{rem}
\begin{cor}\label{smoothing}
If $E, C$ are as in Lemma \ref{balanced-lem}
then a general smoothing of $(E, C)$ is balanced.
\end{cor}
\begin{proof}
	Note that $\check E\otimes E$ is a direct sum of lines bundles
	whose total degree on any subchain
	(a fortiori, on any component) is $\geq -1$. It is easy to see 
	by induction on chain length that such a line bundle has $H^1=0$.
Now consider a deformation $(\tilde E, \tilde C)$ of $(E, C)$ where $C$ smooths. 
Let $(\tilde E_1, \tilde C)$
be an analogous deformation where $\tilde E_1$ is a direct sum of line bundles deforming the
RHS of \eqref{upper}. It is easy to check that 
\[H^1(C, (\check{\tilde E}_1\otimes \tilde E)\otimes\O_C)=H^1(C,\check E\otimes E)=0\]
and consequently the isomorphism \eqref{upper} extends to a map $\tilde E_1\to\tilde E$,
which must be an isomorphism on the general fibre.
\end{proof}
\begin{rem}
	The line bundle $\O(1, -1, -1)$ has $H^1\neq 0$.
	\end{rem}
We conclude with some remarks on bundles on combs. 
These will be used in the proof of Theorem \ref{balanced-high} below.
By definition
a comb is a rational tree of the form $B\cup\bigcup\limits_{i=1}^t T_i$
where $B$, the base, and $T_1,..., T_t$, the teeth, 
are $\P^1$ and each tooth meets the base. 
For 'generalized comb', replace $\P^1$
by 'rational tree' ( thus every rational tree is a generalized comb but
it is understood that a particular pettinal presentation is specified).
In general, the behavior of bundles on combs is complicated and we will
just consider a special, dentally challenged case used below
(see Theorem \ref{balanced-high}).
\begin{lem}\label{comb-lem}
	Let $E$ be a rank-$r$ bundle on a generalized comb $B\cup\bigcup\limits_{i=1}^t
	T_i$. Assume\par (i) $E_B, E_{T_1}, ...E_{T_t}$ are strongly balanced;
	\par (ii) $r_+(E_{T_i})=r-1, i=1,...,t$
and the upper hyperplanes in $V_+(E_{T_i})\subset E(p_i), i=1,...,t$ are general,
where $p_i=B\cap T_i$;\par
(iii) $t\leq r_+(E_B)$ (resp. $t<r_+(E_B)$).\par
Then $E$ is balanced (resp. strongly balanced).
	\end{lem}
\begin{proof}
	We assume $t<r_+(E_B)$ as the equality case is similar.
	In this case we can see easily as above that
	$E$ is  strongly balanced with unique upper subbundle 
	or rank $r_+(E_B)-t$ and degree $d_+=\deg_+(E_B)+
	\sum\limits_{i=1}^t\deg_+(E_{T_i})$ 
		corresponding to
	$V_+(E_B)\cap\bigcap\limits_{i=1}^tV_+(E_{T_i})$, complemented by
	$t$ many general line subbundles of degree $d_+-1$ corresponding to 
	$V_+(E_B)\cap\bigcap\limits_{i\neq j}V_+(E_{T_i}), j=1,...,t$
	(mod  $V_+(E_B)\cap\bigcap\limits_{i=1}^tV_+(E_{T_i})$), 
	glued to a general, non-upper subbundle of $E_{T_j}$,
	plus $r-r_+<r-t$ many general line subbundles
	corresponding to linearly independent elements of
	 $\bigcap\limits_{i=1}^tV_+(E_{T_i})$ glued to general,
	non-upper, line subbundles on $E_B$.
		\end{proof}

\section{Rational curves with balanced normal bundle}
The normal bundle $N_{C/X}$ of a rational curve $C$ on $X$ is related
to the movement of $C$ on $X$: $N_{C/X}$ being semipositive,
i.e. globally generated, means that the curve moves freely,
filling $X$; when that is so, $N_{C/X}$ being balanced means
$C$ can be 'pinned down' at the maximum number of general points on $X$,
i.e. that $C$ has 'maximal momentum' (or energy).\par
Our purpose in this section
is to prove, by a suitable degeneration,
 the existence of rational curves $C$ of low degree on general Fano hypersurfaces $X$
such that the normal bundle $N_{C/X}$ is balanced. The idea is to
work with a polygon, i.e. a chain of lines, on a union of hyperplanes, with at most one line
per hyperplane, and suitably modify its
normal bundle at non-lci points, which are the smooth points on the polygon that are
singular on the union of hyperplanes. This result, which will be used in the proof 
of our  irreducibility
result. will later be generalized to higher-degree curves
using fans (see Theorem \ref{balanced-high}). No fans will be used here.\par
We begin with a local construction. In $\C\times \C^n, n\geq 3$ 
with coordinates $t, x, y, z, ...$ consider
a hypersurface $U$  with equation
$tz=xy$. We view $U$ as a family of hypersurfaces $U_t\subset\C^n$ over $\C$.
Then $U_t$ has a normal-crossing
double point for $t=0$ and is smooth for $t\neq 0$. Now blow up $U$ in the locus $t=y=0$, i.e. one
component $U_{0,1}$ of the special fibre $U_0$, which has coordinates $x, z,...$. 
The blowup $\tilde U$ is smooth and the map $\tilde U\to U$ is small. In the relevant open set we can write
\[t=uy, x=uz,\]
so the special fibre $\tilde U_0$ has two components $\tilde U_{0,1}, \tilde U_{0,2}$, 
with respective equations $y, u$,  and  $\tilde U_{0,2}\to U_{0,2}$ is an isomorphism while
$\tilde U_{0,1}\to U_{0,1}$ is the blow-up of the  smooth codimension-2 locus $x=z=0$.\par
Now working globally,  let $H_i=(x_i)$ be the $i$-th coordinate hyperplane in $\P^n$ and let
$X_0=\bigcup\limits_{i=1}^d H_i=(x_1...x_d)$, $d\leq n$. Consider an $e$-gon, $e\leq d$, $C_0\subset X_0$ of the form
\[C_0=\bigcup\limits_{i=1}^e L_i, L_i\subset H_i, L_i\cap L_{i+1}=L_i\cap H_{i+1}=: p_{i,i+1}, i=1,...,e-1,\]
$L_i\subset H_i$ being general lines. Set $p_{i,j}=L_i\cap H_j, K_{i,j}=H_i\cap H_j,  i\neq j$. Thus each point $p_{i,j}, j\neq i+1$,
is a smooth, non-lci point on $C_0$ relative to $X_0$,
while each $p_{i,i+1}$ is an lci point relative to $X_0$.
Note that $K_{i,j}$ can be identified with the fibre at $p_{i,j}$ of the projectivized normal bundle
$\P(N_{L_i/H_i})$ which is a trivial projective bundle $L_i\times K_i$, so we identify $K_{i,j}\simeq K_i, \forall j\neq i$.
\par
Now let $X_1=(f_d)$ be a general degree-$d$ hypersurface through
all the non-lci points $p_{i,j}, j\neq i+1$. By Lemma \ref{gen-position-lem} below,
$X_1$ is smooth at all the non-lci points. Consider a family $\cX$ in $\C\times\P^n$ with equation
$\prod\limits_{i=1}^d x_i+tf_d$. Construct  the modification $\tilde \cX$ as above where
at each non-lci point $p_{i,j}$, only $H_i$ is blown up (i.e. $H_i$ plays the role of $U_{0,1}$
and $f_d$ plays the role of the $z$ coordinate). 
Then $C_0$ lifts to a polygon $\tilde C_0$ in
$\tilde \cX$  and $\tilde \cX$  is
smooth along $\tilde C_0$ and $\tilde\cX_0$ is smooth at the points corresponding
to the $p_{ij}, j\neq i+1$ because the proper transform of $L_i$ is disjoint from the inverse
imge of $H_j$. Thus $\tilde C_0$ is everywhere lci relative to $\tilde\cX_0$.
\begin{lem}\label{gen-position-lem}
A general $X_1$ as above is smooth in all $p_{i,j}, j\neq i+1$ and transverse to $H_j$.
Moreover, for each fixed $i$, if we identify $K_{i,j}\simeq K_i,  \forall j\neq i+1$, 
then the collection of tangent
planes to $X_1\cap K_{i,j}$ at the $p_{i,j}, j\neq i+1$ 
become identified with a general collection
of hyperplanes in $K_i$.
\end{lem}
\begin{proof}
Use induction on $d$ and $e$, using the exact sequence
\[\exseq{\O(d-1)(-\sum\limits_{\stackrel{i\geq 2}{j\neq i+1}} -p_{i,j})}
{\O(d)(-\sum\limits_{\stackrel{i\geq 1}{j\neq i+1}} -p_{i,j})}
{\O_{H_1}(d)(-\sum\limits_{j\neq 2} -p_{1,j})}.\]
By induction, the first term has vanishing $H^1$, hence so does the middle term and the map
on $H^0$ is surjective.
\end{proof}
%We may identify the normal bundle $N_{L_i/H_i}$ with $V_i\otimes\O(1)$ for a suitable vector space $V_i$.
%Moreover $\P(V_i)$ is naturally isomorphic tp $\P(T_{p_{i, i\pm 1}}(H_i\cap H_{i\pm 1})$ and therefore
%the projectivizations $\P(V_1),...,\P(V_e)$ are all naturally isomorphic.
At a point of $\tilde C_0$ coming from a 
non-lci point $p_{i,j}$, the normal bundle $N_{\tilde C_0/\tilde X_0}$ is an elementary 
'down' modification
(i.e. subsheaf of colength 1) of $N_{L_i/H_i}$ corresponding to the hyperplane
 $T_{i,j}=T_{p_{i,j}}X_1\cap T_{p_{i,j}}(K_{i,j})\subset 
 K_{i,j}\simeq K_i=N_{L_i/H_i}(p_{i,j})$,
 i,e, the kernel of the surjection $N_{L_i/H_i}\to N_{L_i/H_i}(p_{i,j})/T_{i,j}$
From the Lemma and the results of the previous section
 it then follows easily that the normal bundle $N_{\tilde C_0/\tilde X_0}$ is balanced. In
particular, it is unobstructed and $\tilde C_0$ smooths out in the family $\cX$ to a smooth curve on $X_t$
with balanced normal bundle. Since a balanced bundle of nonnegative degree is semipositive,
hence has vanishing $H^1$, a rational curve  with  balanced normal bundle on a 
smooth Fano hypersurface automatically deforms with the hypersurface, we conclude:
\begin{prop}\label{balanced-prop}
For all $e\leq d\leq n$, a general hypersurface $X$ of degree $d$ in $\P^n$ contains
a rational curve $C$ of degree $e$ such that the normal bundle $N_{C/X}$ is balanced:
\[N_{C/X}\simeq r_+\O_C(a_++1)\oplus (n-2-r_+)\O_C(a_+)\]
where
\eqspl{balanced-normal}
{a_+=\lfloor \frac{(n+1-d)e-2}{n-2}\rfloor, r_+=(n+1-d)e-2-(n-2)a_+.}
\end{prop}
This result will later be reproved and generalized to higher-degree curves 
(see Thm \ref{balanced-high}). The lower-degree case will play a role in the proof
of our main result, through the following consequence:
\begin{cor}\label{curve-thru-point-cor}
Notations as above, if $a_+\geq 0$ and $q_1,...,q_{a_++1}$ are general points on $X$, 
there is a rational curve $C$ of degree $e$ on $X$ through them and the family
%and $W_i\subset N_{C/X}\otimes\C(p_i)$ is a general subspace
%of codimension $r$ for $i=1,...,r_+$,
 of such curves near $C$ passing through $q_1,..,q_{a_++1}$
 is $r_+$-dimensional and fills up a 
 subvariety whose proper transform on the blowup of $X$ in
 $q_1,...,q_{a_++1}$ is locally  smooth and
 $(r_++1)$-dimensional  at its intersection with each exceptional divisor. 
 %whose
%upper subspace at $p_i$ meets $W_i$ nontrivially is  nonzero.
\end{cor}
\begin{proof}
For $C$ with balanced normal bundle we have  $H^1(N_{C/X}(-p_1...-p_{a_++1}))=0$. 
Therefore by deforming $C$ and the points we get
a general $(a_++1)$-tuple on $X$ and moreover for any fixed collection of points $q_i$
the curves sweep up a subvariety whose tangent spaces on the blowup come from
$H^0(N_{C/X}(-p_1...-p_{a_++1}))=H^0(N_{C/X, +}(-p_1...-p_{a_++1}))$, 
where $N_{C/X, +}(-p_1...-p_{a_++1})$ is
 a trivial bundle.
\end{proof}
\begin{rem}
The variety filled up by the curves through the $(q_\bullet)$ points will 
naturally be singular
at those points, so its dimension cannot be computed by looking
at tangent spaces.; what is important for our arguments below 
is its local dimension at the points.
\end{rem}
\begin{rem}
In general, there exists rational curves with positive, hence unobstructed, but unbalanced
normal bundle on general hypersurfaces. 
For example, a quadric $X$ of large enough dimension is ruled by planes,
and conics in those planes- which are special among conics on $X$- have positive, unbalanced normal
bundle.
\end{rem}
\begin{rem}
After this was written, we found that closely related results
on balanced normal bundles has been obtained earlier by
Coskun and Riedl \cite{coskun-riedl} using different, degeneration-free methods.
Their results apply in a larger range of degrees compared to Lemma \ref{balanced-prop}
 and for complete intersections as well,
rather than just hypersurfaces. Anyhow a more general result for hypersurfaces will be given
below in Proposition \ref{balanced-high}
\end{rem}
Next we briefly review with proof, in a form convenient for our purposes, 
some facts on
unfolding and unkinking rational curves,
which can be found in Koll\'ar's book \cite{kollar-rat-curves},
especially Sect. II.3, Thm. 3.14 on free rational curves.
Let \[f:\P^1\to X\]
be a morphism to a smooth $n$-fold, which we assume is a 
general member of a maximal family filling up $X$.
\begin{lem}\label{free-rational-lem}
Notation as above, $f$ is an immersion of $n\geq 2$ and an embedding if $n\geq 3$.
\end{lem}
\begin{proof}
Recall that a bundle on $\P^1$ is semipositive when it is 
a direct sum of line bundles of nonnegative degrees.
Note that our filling hypothesis implies- indeed is equivalent to - the property that $f^*(T_X)$
is semipositive. We show $f$ is
injective if $n\geq 3$. Suppose for contradiction $p_1\neq p_2\in\P^1$ are such that
$f(p_1)=f(p_2)=x$. Then, for any first-order deformation of $f$, the associated  section $s$
of $f^*(T_X)$ has the property that $s(p_1)=s(p_2)\in T_{X}(x)=T_{X, x}\otimes\C(x)$. Let $L$ be any hyperplane in 
$T_{X}(x)$ containing $df(T_{\P^1}(p_1))+df(T_{\P^1}(p_1) $. There is an associated colength-1
subsheaf $M\subset f^*(T_X)$ such that $f^*(T_X)/M$ is the skyscraper sheaf $T_X(x)/L$ at $x$,
and then $s\in H^0(M)$. But since $M$ is a direct sum of line bundles of degree $\geq -1$,
we have
\[h^1(M)=0=h^1(f^*(T_X)), \ \ h^0(M)<h^0(f^*(T_X)).\]
This contradicts the fact that $f$ is general in an unobstructed family of dimension
$h^0(f^*(T_X))$.\par
The argument that $f$ is unramified in $n\geq 2$ proceeds similarly: if $f$ is ramified at 
$p\in\P^1$, there exists $k\geq 2$ and an injection $\sym^k(T_{\P^1}(p))\to T_X(f(p))$
such that the value at $p$ of any first-order deformation of $f$ is contained in the image.
Then we can take for $L$ any hyperplane containing the image and argue as above.
\end{proof}
\section{Quasi-cones}
A quasi-cone in $\P^n$ is by definition a hypersurface  with equation
\[F=Q(F_d, F_{d-1})=F_d(x_1,...,x_n)+x_0F_{d-1}(x_1,...,x_n)\]
where $x_0,...,x_n$ are homogeneous coordinates and  $F_d, F_{d-1}$ are homogeneous polynomials
of degree $d, d-1$ respectively. 
This is equivalent to saying $F$ has degree $d$ and multiplicity $d-1$ at $p=[1,0,...,0]$.
Note that $F_{d-1}$ is uniquely determined by $F$ while $F_d$
is uniquely determined mod $F_{d-1}$. 
Let \[b:\tilde P\to\P^n\] denote the blowup of $p$,
with exceptional divisor $E\simeq\P^{n-1}$, and let
\[\pi:\tilde P\to\P^{n-1}\]
denote the natural projection, which is an isomorphism on $E$.\par
Now let $\bar X\subset\P^n$ be a quasi-cone with equation $Q(F_d, F_{d-1})$ as above, and let
\[b_X:X\to\bar X\]
be the proper transform of $\bar X$ in $\tilde P$, with exceptional divisor $E_X=X\cap E \simeq F_{d-1}$.
Via $\pi_X$, $X$ can be realized as the blowup of $\P^{n-1}$ in $Y=V(F_{d-1})\cap V(F_d)$.
In fact the rational map $b_X\circ\pi_X\inv:\P^{n-1}\cdots\to\P^n$ corresponds to the linear system
on $\P^{n-1}$ with base locus $Y$ generated by $x_1F_{d-1}, ...,x_nF_{d-1}, F_d$.
$\bar X$ is said to be \emph{quasi-smooth}
 if $Y$ is smooth.
In this case it is easy to see that $X$ is smooth.
\par Let $\tilde Y\subset X$ be the exceptional divisor of $\phi_X$, which is isomorphic
to the projectivized conormal bundle of $Y$ in $\P^{n-1}$, and maps to the ruled subvariety $\bar Y\subset\bar X$
defined by $F_{d-1}, F_d$, contracting
 the 'infinity section' $\tilde Y\cap E_X$.\par
Now let $L, H$ be the line bundles on $X$ pulled back respectively from
the hyperplane bundles on  $\P^n, \P^{n-1}$. Then we have
\[L=dH-\tilde Y,\]
\[\tilde F=(d-1)H-\tilde Y,\]
and consequently
\[L=\tilde F+H.\]
The following is worth noting
\begin{lem}
If $d\leq n-1$  and $\bar X$ is quasi-smooth then $X$ is Fano; in fact, $-K_X$ is very ample.
\end{lem}
\begin{proof}
$-K_X=nH-\tilde Y=(n-d)H+L$. The map
\[b\times\pi:X\to \P^n\times\P^{n-1}\]
is clearly an embedding, hence $-K_X=(b\times\pi)^*(\O(n-d, 1))$ is very ample.
\end{proof}
Now let $C$ be an irreducible  curve in $X$ corresponding to a curve in $\bar X$ of degree $e$ 
and multiplicity $m$ at $p$. Then
\[C.L=e, C.\tilde F=m.\]
Consequently, 
\[C.H=e-m, C.E=(d-1)(e-m)-m.\]
Thus, $C$ maps to a curve $C_*$  in $\P^{n-1}$ that has degree $e-m$ and meets $Y$ in $(d-1)(e-m)-m$ points
(among the $(d-1)(e-m)$ intersections of $C_*$ and $F_{d-1}$).

For instance, when $m=e-1$, $C_*$ is a $(d-e)$-secant line of $Y$. \par
Next, we make an elementary remark about curves in $\tilde Y$, considered as 'infinitely-near
secants' to $Y$. This result will not be needed later, but it is psychologically important.
To begin with note that 
\[\tY=\P(\check N_Y)=\P(\O_Y((-d+1)H, -dH))\]
and that
\[\O_{\tY}(\tY)=\O_{\tY}(-1)=\O_{\P(\check N_Y)}(-1).\]
Now let $C\to \tilde Y$
be an irrecudible curve. Such a curve corresponds to a map $f:C\to Y$ together
with an invertible quotient $f^*(\check N_Y)\to A$, and clearly
\[A=\O_{\tY}(1).C.\]
Now clearly the smallest-degree invertible quotient of $\check N_Y$ is $\O(-dH)$ which corresponds
to the 'infinity section' $\tY\cap\tilde F$. For any curve not contained in the infinity section,
the degree of the quotient $A$ is at least  $(-d+1)\deg(H)$. Thus we conclude
\begin{lem}[Infinitely-near secants]\label{inf-near-sec-lem}
For an irreducible curve $C\to\tY$, 
not contained in $\tilde F$, we have
\[C.\tY\leq (d-1)C.H.\]
If $C$ is  contained in $\tilde F$, we have
\[C.\tY= dC.H.\]
\end{lem}
\section{Fans and their hypersurfaces}
In this section we fix a dimension $n$. Let $\tilde\P^n$
denote the blowup of $\P^n$ at a point.
\begin{defn}
	A fan (of fibre length $m$ or $m$-fan) is a normal-crossing variety of the form
	\[P_1\cup...\cup P_m\]
	where $P_1\simeq...\simeq P_{m-1}\simeq\tilde\P^n$, $P_m\simeq\P^n$
	and for $i=1,...,m-1$,
	 $P_i\cap P_{i+1}$ is the exceptional divisor in $P_i$ and
	a hyperplane in $P_{i+1}$ not meeting its exceptional divisor if any.\par
	A relative fan is a flat proper morphism $\cP\to T$ whose fibres are fans
	(NB the fibre length of these fans may vary). 
	\end{defn}
Our interest in fans stems from the fact that they are degenerate forms of $\P^n$,
as illustrated by 
the following construction of a relative fan with general fibre a 1-fan 
(i.e. $\P^n$) and special fibre a 2-fan.
Let $(B, 0_B)$ be a smooth affine pointed curve
(e.g. $\A^1$) with coordinate $t$ and let $\cP(1)/B$ denote the blow-up 
of $\P^n\times B$ in $(0,0_B)$.
The fibre of $\cP(1)$ over $0_B\in B$, 
is a 2-fan $P_1\cup P_2$ where $P_1$ is the blowup of $\P^n\times 0_B$ in $(0, 0_B)$,
with exceptional divisor $E_1\simeq\P^{n-1}$ and $P_2$ is
the exceptional divisor, which is a copy of $\P^n$ containing
$E_1=P_1\cap P_2$ as a hyperplane. Thus $\cP(1)/B$ itself is a relative fan
(of max fibre length 2).  We will also denote by $E_0\subset P_1$ the pullback of a general hyperplane 
from $\P^n$.\par
 The above construction can be iterated
to yield for all $M\geq 1$ a relative fan $\cP(M)/T=B^M$ called a \emph{standard fan}
with smooth total space and
 fibres $m$-fans for $1\leq m\leq M+1$. Inductively,
$\cP(M+1)$ is the blowup of $\cP(M)\times B$ in $0_M\times B^M\times 0_B$ where
$0_M$ is a point in the 'last component'  $P_{M+1}$ of $\cP(M)$ over $0_B^M$ which is
not in $E_M$.  An $m$-fan fibre of $\cP(M)$ is determined up to locally trivial deformation
(and even up to isomorphism) by a sequence $(i_1<...<i_m)\subset [1, M+1]$ and has the form
$P'_1\cup...\cup P'_m$ where each $P'_j$ is a smoothing of $P_{i_j}\cup...\cup P_{i_{j+1}-1}$
for $j<m$ or $P_{i_m}\cup...\cup P_{M+1}$ for $j=m$.  We call this a fibre of type $(i_\bullet)$.  \par
Note that the base $B^M$ of $\cP(M) $ is endowed with a natural coordinate
stratification so that over the $i$th coordinate
hyperplane $t_i=0$, the $i$-th 'joint' $E_i$ forms a flat family
whose general fibre is a 2-fan with double locus $E_i$. 
The pullback of $t_i=0$ is a sum of 2 divisors $D_i^-, D_i^+$
whose fibres over a given point consist of all fan components
'below' or 'above' $E_i$ respectively.
The fact that $E_i$ has opposite normal bundles in $P_i$ and $P_{i+1}$
is a direct consequence of the smoothness of the total space of the relative fan $\cP$:
indeed we have
\[ N_{E_i/P_i}=N_{P_{i+1}/\cP}|_{E_i}, N_{E_i/P_{i+1}}=N_{P_{i}/\cP}|_{E_i},\]
however, if $F$ is a complete fibre, we have
\[N_{P_{i}/\cP}|_{E_i}\otimes N_{P_{i+1}/\cP}|_{E_i}=
N_{F/\cP}|_{E_i}=\O_{E_i}.\]

There is a useful operation on relative fans called  \emph{stretching}.
Given a relative fan $\cP/T$ with smooth total space,
plus a component 
$E_i$ of the  relative double locus,
living over a divisor $(t_i)\subset T$, the stretch $\cP^+_i/T^+$  
is defined as follows. Consider a base change $T^+\to T$
given locally by $s_i^2=t_i$. 
The pullback of $E_i$ becomes
singular, with local equation $s_i^2=xy$. 
Blowing this up, the total space $\cP^+_i$ becomes
smooth and the pullback of $E_i$ is a conic bundle 
$P_i'$ meeting the proper transforms of $P_i$
and $P_{i+1}$ in disjoint sections denoted $E'_i, E"_i$ whose respective normal
bundles in $P'_i$ are $\O_{E'_i}(1), \O_{E"_i}(-1)$. This easily implies that
\[P'_i=\P_{E'_i}(\O(1) \oplus\O)=\P_{E"_i}(\O\oplus\O(-1))\]
so again $P'_i$ is a $\P^n$ blown up at a point 
and $\cP^+_i/T^+$ is a relative fan,
with max fibre chain length increased by 1.\par
The stretch $\cP^+_i/T^+$
just constructed is a subvariety of a relative fan with larger base called
a \emph{big stretch} of a relative fan $\cP/T$ constructed 
as follows. Consider the base change \[T^{++}\to T, \ \  t_i=t'_it_i",\]
which has, locally over $E_i$ in the above setup, the equation \[t'_it"_i=xy.\]
Then blow up the Weil divisor defined 
locally $(t'_i, x)$ (geometrically and globally
this consists of the sum of the components of 
the pullback family over $(t'_i)$ lying to one side
of $E_i$).  This yields a new relative fan
 $\cP^{++}_i/T^{++}$ with smooth total space which
contains the regular stretch $\cP^+_i/T^+$ as the divisor
with equation  $t'_i=t_i"$. The same relative fan is obtained by blowing up
$(t"_i, y)$. The parameter 
$t'_i$ (resp. $t"_i$) cuts out on $\cP^{++}/T^{++}$ the divisor 
$E'_i$ (resp. $E"_i$).
It is easy to check that a big stretch of a standard relative fan $\cP(m)/B^m$ is $\cP(m+1)/B^{m+1}$.\par
The following result shows that every relative fan is essentially standard:
\begin{prop}\label{curve-fan-prop}
Locally at each point of $B^m$, $\cP(m)$ is a versal deformation of its fibres. At $0_B^m$, $\cP(m)/B^m$
is miniversal.
\end{prop}
\begin{proof}
Clearly locally trivial deformations of a fan $P$ are trivial, i.e. $H^1(P, T^0_P)=0$ and likewise for $H^2$. 
On the other hand the $T^1$ sheaf is clearly $\bigoplus \O_{E_i}$ because the singularity
of $P$ along each $E_i$ is just an ordinary curve double points times $E_i$.
This shows that the natural 'Kodaira-Spencer' map $T_tT\to H^0(P, T^1_P)$ is surjective and even an isomorphism
if $t=0_B^m$. This proves our assertion.
\end{proof}
%\section{Hypersurfaces on fans}
Next we discuss hypersurfaces on fans.
Let $(d_\bullet)=(d_1\geq d_2\geq...\geq d_m\geq 0)$ be a 
decreasing sequence of integers. Consider the line bundle
on $\cP(m)$:
\[\O(d_1,...,d_m)=\O(d_1D_1^+-\sum\limits_{i=2}^md_iD^+_i).\]
This is a line bundle whose restriction on each fibre component $P_i$ is the line bundle
$\O(d_iE_i-d_{i+1}E_{i+1})$ or $\O(d_m)$ if $i=m$. 
On the generic fibre, it is just $\O(d_1)$. Note that this bundle
has no higher cohomology on any fibre, hence locally constant sections:
in fact it suffices to prove this for the most special fibre, 
an $m$-fan, and the latter can be easily done by induction
on $m$, writing the $m$-fan as $P_1\cup(P_2\cup...\cup P_m)$
and using a Mayer-Vietoris sequence.
Also this bundle is clearly very ample if the sequence 
$(d_\bullet)$ is strictly decreasing. 
Finally, note that the restriction
of $\O(d_\bullet)$ to a fibre of type $(i_\bullet)$ can be identified with $\O(d_0,d_{i_\bullet})$.
\section{Curves on fan hypersurfaces}
Let $\pi:\cP\to B$ be a relative fan and let $\cX\subset\cP/B$ be a relative hypersurface  not containing any fibre component of $\pi$. 
We denote by $\cM_0(\cX, e)$
the Kontsevich space of stable genus-0 maps to fibres of $\cX/B$ having the homology class
of a degree-$e$ curve on a general fibre (which equals the homology class of a degree-$e$ curve
on the bottom component of any fibre that is disjoint from all other components, that is, $e(D_0^+)^{n-1})$. 

\begin{prop}\label{no-E-prop}
Assume the fibre $X_0$ of $\cX/B$ over $0\in B$ is a general hypersurface of type $(d.)$
of the fan $P_\bullet=\pi\inv(0)$,
let $\cM_*$ be a component of $\cM_0(\cX,e)$ that surjects to $B$,
and let $f:C\to X_0$ correspond to a general member of the fibre of $\cM_*$ over $0$.
Then no component of $C$ maps with nonzero degree to the singular locus of $X_0$.
\end{prop}
\begin{proof}
We may assume $P_\bullet$ is a 2-fan as the general case is similar. If $C_1$ is a component
of $C$ mapping nonconstantly to the singular locus $E$ of $X_0$, then the local  ideal of $E$
vanishes identically to some order $m$  on $C$. Then by a stretch operation we may reduce
$m$ successively until $m=0$. 
Note that in the stretch operation the limit curve may acquire new components, but these
do not map to the singular locus of the new fan: indeed if $xy$ is a local equation of the special fibre,
where $(x, y)$ are equations of the singular locus, extra components
arise at points where the function $x/y$ is indeterminate on the curve family, so after 
blowing up this function becomes locally regular and \emph{nonconstant}, which means
new components do not map into the singular locus.\par
Thus, after sufficient stretching of the relative fan,
 $C_1$  maps  nonconstantly into some component of a fan fibre other than the bottom one.
Then by going to a versal deformation of the special fan, we may consider a general
2-fan fibre and there $C_1$ still maps nonconstantly. This implies $f$ could not be constant to begin with. 
\end{proof}
\section{Multisecant rational curves of complete intersections}\label{secants}
Motivated by the construction of the preceding section, we now consider families of
multisecant rational curves
of general complete intersections.
For simplicity, we have stated only a particular codimension-2 case,
namely that of a $(d, d-1)$ complete intersection,
which is the one we need,
but the result and the proof are valid in greater generality.
We begin with the case of lines. 
\begin{prop}
Let $Y$ be a general $(d, d-1)$ complete intersection in $\P^{n-1}, n\geq 4 $ 
and let $a\leq d-1\leq n-2$. Then\par 
(i)  the locus $\Sec_a(Y)$ of  $a$-secant lines to $Y$ is 
irreducible reduced $(2n-4-a)$-dimensional, 
has a smooth normalization, 
and its general member meets $Y$ transversely in exactly $a$ points.\par
(ii) For a general point $q\in\P^{n-1}$, 
 the locus $\Sec_a(Y, q)$ of  $a$-secant lines to $Y$
 through $q$ is  $n-2-a$-dimensional and nonempty.

%(ii) the $a$-secant lines  to $Y$ fill up $\P^{n-1}$;\par
%(iii)  the set of components of the locus of $a$-secants
%is acted on transitively by the  monodtromy as $Y$ varies in the space of complete intersections.
\end{prop}
\begin{proof} 

Let $U_d$ be the space of pairs of hypersurfaces $(F_{d-1}, F_d)$ such that $Y=V(F_{d-1}, F_d)$ has at most 
one singular point at which $F_{d-1}$ and $F_d$ are smooth and simply tangent
(i.e. have same linear terms and transverse quadratic terms). 
This is a 'big' open subset of $\P^{N_{d-1}}\times\P^{N_d}$,
i.e. the complement
of $U_d$ has codimension $>1$, and hence $U_d$ is simply connected.
Consider the locus $W_d$ of quadruples $(L, A, F_{d-1}, F_d)$
such that  $(F_{d-1}, F_d)\in U_d$, $L$ is a line, and
$A$ is a divisor of degree $a$ on $L$ contained in $Y=V(F_{d-1}, F_d)$.
This is clearly an irreducible nonsingular variety of dimension $2n-4+N_{d-1}+N_d-a$,
and the projection $\pi_d: W_d\to U_d$ is proper. Moreover by working over a fixed $L$
note that the subset of $W_d$ where $Y$ is singular along $L$ has codimension $>1$.
It follows from the Lemma below that 
$\pi_d$ is also surjective; alternatively, since the tangent spaces of $Y$ can be specified
generally $\mod L$ at $A$ for a general element of $W_d$ (see 
Lemma \ref{secant-balanced-chain-lem}), it follows that $\pi_d$
is generically smooth, hence surjective.

 Our claim is that a general fibre of $\pi_d$ is irreducible, and we argue by contradiction. 
 To begin with,
 note that by irreducibility of $W_d$, the components of such a general fibre are
 monodromy-interchangeable. Also, again by irreducibility of $W_d$, 
  if $Z$ is any codimension-1 subvariety of $U_d$ then
the dimension   of the fibre of $\pi_d$ over
any general point $\ul F \in Z$ is still $n-2-a$, same as the general fibre dimension.
Now if a general fibre of $\pi_d$ is reducible, consider a Stein 
 factorization of $\pi_d$. Then by simple connectivity of $U_d$,  
there would be a codimension-1 locus $Z$ in $U_d$ over which the fibre of $\pi_d$
has a nonreduced component, whose support is algebraically equivalent to a component
of a general fibre over $U_d$, and hence to a fraction of an entire general fibre.
By dimension counting, for  a general point of such a fibre
component,  $Y$ is smooth
at its intersection with $L$. Moreover 
a general line $L$ in the fibre goes through  a general point of $\P^n$
relative to $Y$. Consequently
by \cite{filling},  $L$ is a good (unobstructed) secant to $V(F_{d-1}, F_d)$, i.e. a smooth point of the fibre over $U_d$.
This contradiction proves assertion (i).\par
Assertion (ii) now follows from the fact that for $a\leq d-1$ the family of $a$-secant lines to $Y$ 
fills the ambient projective space, which is a consequence of the following
easy result: 

\end{proof}
\begin{lem}
Let $Y\subset \P^n$ be a codimension-2 complete intersection of type $(c,d)$ and let $a\leq\min(c,d)<n$.
Then there is an $a$-secant line to $Y$ through a general point of $\P^n$.
\end{lem} 
\begin{proof} (One of many).
Working by induction on $a$, it suffices to prove that in the indicated
range, the family of $(a+1)$-secants 
is nonempty and has codimension 1 in that of $a$-secants, 
which would follow if the family of $(a+1)$ secants has nonempty intersection
with a general curve section of the family of $a$-secants.
It thus suffices to prove that for $a<\min(c,d)<n$ any 1-parameter
 family of $a$-secants contains an $(a+1)$-secant.
To that end, note the tautological
$\P^1$-bundle over such a 1-parameter family maps 
to a cone in $\P^n$ whose  desingularization we denote by $K$. 
If there is no $(a+1)$-secant in our family,
then the intersection points of the rulings 
with $Y$ trace out a divisor $R$ disjoint from the section
that contracts to the vertex of the cone, hence $R$ is linearly equivalent
to $aH$. Now $Y$ corresponds to a section of $\O(c, d))$ whose
pullback minus $R$ yields a nowhere vanishing section of $\O(c-a, d-a) $ on $K$. But this clearly
cannot exist.
\end{proof}
Our aim now is to extend the Proposition from the case of lines to that of rational
curves and some of their limits. By a \emph{simply covered rational curve}
we mean a curve $C$ admitting a surjective map from a rational tree $\tilde C$, which has fibre degree 
1 over a general point of each component of $C$. Thus, no component of $\tilde C$ may
map with degree $>1$, though some may map to a point, and no two components
of $\tilde C$ may map to the same curve. Clearly, any simply covered rational curve is smoothable in $\P^n$.
\begin{thm}\label{secant-thm}
Let $Y$ be a general $(d, d-1)$ complete intersection in $\P^{n-1} $ and let $a\leq e(d-1)$,
$d\leq n-1$
and $e\leq d-1$. Then
 the locus $\Sec_a^e(Y)$ of  .simply covered 
 rational curves of degree $e$ that are
 $a$-secant   to $Y$ is an irreducible reduced $((e+1)n-4-a)$-dimensional 
and has a smooth normalization, 
and its general member is an irreducible nonsingular
curve that meets $Y$ transversely in exactly $a$ points
and passes through a general point of $\P^{n-1}$.\par
%???(ii) For a general point $q\in\P^{n-1}$, 
% the locus $\Sec_a^e(Y, q)$ of  
% simply covered rational curves of degree $e$ through $q$,
%  $a$-secant   to $Y$, is an irreducible reduced $en-2-a$-dimensional 
% $a$-secant lines to $Y$
% through $q$ is an irreducible reduced $n-2-a$-dimensional 
%and has a smooth normalization, 
%and its general member meets $Y$ transversely in exactly $a$ points.\par
%%(ii) the $a$-secant lines  to $Y$ fill up $\P^{n-1}$;\par
%(iii)  the set of components of the locus of $a$-secants
%is acted on transitively by the  monodtromy as $Y$ varies in the space of complete intersections.
\end{thm}
\begin{proof}
To begin with, we use the same construction as
in the case of lines above  and let $W^e_d$ denote the set of triples $(C, A,\ul F)$ such that 
$\ul F\in U_d$, $C$ is a smooth rational curve of degree $e$  and $A$ is a subscheme 
of length $a$ of the schematic intersection $C\cap Y$ (notations as above ).
Note that a general $C$ will be a rational normal curve in its linear span which is a $\P^e$. 
Let $\bar W_d^e$ denote the closure of $W^e_d$ in the product $ \H\times U_d$,
where $\H$ is the relative Hilbert scheme
(of zero-dimensional subschemes of length $a$) 
of the universal (stable) curve over  
$\cM$,  the Konstevich space of stable genus-0 unpointed maps of degree $e$  to $\P^{n-1}$.
Points of $\bar W^e_d$ have the form $(f, C, A, F_{d-1}, F_d)$ such that
$f: C\to\P^{n-1}$ is a stable map and $A$ is a length-$a$ subsceme of $f^*(F_{d-1}, F_d)$,
and the whole tuple is a limit of similar ones where $f$ is an embedding of a smooth rational 
(normal) curve.
By a result of Gruson et al. \cite{glp}, $W_d^e$ 
is irreducible of dimension $(e+1)n-4+N_{d-1}+N_d-a$.
For the general point $(C, A, \ul F)\in W_d^e$, we can choose $F_{d-1}, F_d$ to have general
enough tangent spaces at $A$ (this is already true when $C$ is a chain of lines), 
by Lemma
 \ref{secant-balanced-chain-lem}. Consequently,
using Lemma
 \ref{secant-balanced-lem} below, it follows that
 the secant bundle $N^s_C$ has vanishing $H^1$.
This shows that the projection to $U_d$ is 
a smooth morphism at that point (hence it is dominant and has a smooth general
fibre of the expected dimension).
 \par
Note that  $\H$ is smooth \cite{hilb} and there is an
open subset $\H'\subset\H$ whose complement has codimension 3,
such that for $(f, C, A)\in\H'$,
$A$ has length 1 or 0 
at any singular point of  $C$ (if $C$ has any) and at any point
lying in a non-singleton fibre of $f$. Over  $\H'$ the projection from $\bar W$
is an affine bundle. 
 Note that $\bar W^e_d$ contains points where $C$ is an embedded curve of the
 form $L\cup C_{e-1}$ where $L$ is a line and $C_{e-1}$ has degree $e-1$,
 both going through a general point $q\in\P^{n-1}$ with respect to $Y$,
 and $A$ consists of $a_1\leq d-1$ points on $L\setminus q$ 
 plus $a_2\leq(e-1)(d-1)$ points on $C_{e-1}\setminus q$.
  Indeed such secant curves  $L, C_{e-1}$ exist
 by induction and by construction the secant sheaves $N^s_{C_{e-1}}, N_L^s$
 are generically globally generated, hence semipositive, where $N^s_*$ denotes the secant
 sheaf, which may be identified with the normal bundle to the proper transform
 in the blowup of $\P^{n-1}$ in $Y$. 
 Now $N^s_L(q)$ and $N^s_{C_{e-1}}(q)$ admit a common quotient $Q$ such that the
 kernel $K$ of the natural map $N^s_L\oplus N^s_C\to Q$ corresponds to connected, locally
 trivial deformations of $L\cup C_{e-1}$ (compare the proof of Lemma \ref{free-rational-lem}). 
 Since $K$ contains $N^s_L(-q)\oplus N_{C_{e-1}}^s(-q)$,
 it has $H^1=0$.
 Therefore $L\cup C_{e-1}$
 is unobstructed and smoothable as secant. It follows that for a general point in $W^e_d$, 
 $C$ also goes through a general point of the ambient space.
 This also gives another, inductive proof that $H^1(N_C^s)=0$ for $C$ general.
 \par 
 Let $\tilde W$ denote the normalization of $\bar W^e_d$. If a general fibre of $W^e_d$
 over $U_d$ is reducible, then so is a general fibre of $\bar W$ over $U_d$.
 But as $\bar W$ is normal, so is its general fibre hence that fibre is locally irreducible,
 and therefore it must be disconnected. Using as above a Stein factorization
 of $\bar W\to U_d$, it follows again that there is a codimension-1 locus $Z\subset U_d$
whose inverse image in $\bar W$ has a multiple component $\hat Z$, i.e. 
 such that the general fibre of $\bar W$, hence of $W^e_d$ has a multiple component $B$
 whose support is a limit of a general fibre component. Note that because a multiple of $B$ moves,
 i.e. is in a 1-parameter family filling up $\bar W$, 
 $B$ cannot be contractible, i.e. cannot map to a subvariety of codimension $>1$ in $\cM$,
 because that is a topological property.
  Hence $\hat Z$ must map to a locus of codimension 1 or less in $\cM$. 
 \par
 Now if $(C, f)$ is a general element of a codimension
  $\leq 1$ locus in $\M$ then
either\par
 (i) $C$ is smooth and $f$ is an embedding, or\par
  (ii) $C$ is smooth, $n-1=3$ and $f$
is an embedding except of one transverse double point, or\par (iii) $C$ has exactly 2 smooth
components and $f$ is an embedding.\par Now case (ii) is excluded by the assumption
$d\leq n-1, e\leq d-1$, so $C$ would map to a line or conic so we are in fact in case (i) or (iii). 
In case (i), $C$ being irreducible and containing a general point of $\P^{n-1}$, it is
an unobstructed secant to the corresponding $Y$ so the fibre over $U_d$ cannot
be multiple. In case (iii), the locus of curves $(C, f)$ already has codimension 1 so the
$Y$ is general for such $C$; however it is easy to see that for such $(C, f)$ and $Y$ general,
$C$ is an unobstructed secant. This contradiction competes the proof.
 \par
% /********************************************/
 \par

\end{proof} 
\begin{lem}\label{secant-balanced-chain-lem}
Let $C$ be a general chain of $e$ lines in $\P^n, e\leq n$.
Let $Y$ be a general $(d-1, d)$ complete
intersection, $d\leq n$, meeting $C$ in $a\leq e(d-1)$ points with at most $d-1$ on each line. Then the
tangent spaces to $Y$ at $Y\cap C$ can be specified generally modulo
the tangents to $C$.
\end{lem}
\begin{proof}
There is evidently no loss of generality in assuming $C$ is a chain of $n$ lines in $\P^n$
and the intersection $Z=Y\cap C$ comprises $d-1$ points on each line.
We then specialize the chain to a pencil $C=\bigcup L_i$ of $n$ general lines through a point. 
For each $i$, let 
\[G_i=\bigcup \limits_{j=1}^{d-1}H_{i,j}\cup H_i\] 
where $H_{i,j}$ is a general
hyperplane through the $j$th point of $Z$ on $L_i$ while $H_i$ is a hyperplane containing
all the lines $L_j, j\neq i$. Then $G_i$ contains $Z$ and has general tangents on $Z\cap L_i$.
Then a general linear combination $\sum t_iG_i$ also contains $Z$ and has general tangents
at all of $Z$.

%Thus $Y$ can be obtained
%by taking a general $F_{d-1}$ then picking a general $F_d$ through $A=C\cap F_{d-1}$.
%Then our assertion can be easily proved by induction on $n$, writing $C=C'\cup L$ where $C'$
%is contained in a hyperplane $H$ and taking $F_d$ of the form $H\cup F'_{d-1}$ with
%$F'_{d-1}$ general through $F_{d-1}\cap L$.
\end{proof}
Using Lemma \ref{balanced-lem} and, if appropriate, smoothing the chain, we can conclude:
\begin{lem}\label{secant-balanced-lem}
Let $C$ be a general rational curve of degree  $e$ (resp. general chain of $e$ lines)
  in $\P^n, e\leq n$. Let $Y$ be a general $(d-1, d)$ complete
intersection meeting $C$ in $a\leq e(d-1)$ points (resp. with at most $d-1$ on each line).
 Then the associated secant bundle to $C$
relative to $Y$,  viz. $N^s_{C/\P^n}$, is
balanced. 
\end{lem}
Note that by taking $a=e(d-1)$ and putting the curve or chain on a component $X_1$ of a 2-fan
hypersurface $X_1\cup X_2$ of type $(d, d-1)$ (so it is disjoint from $X_2$), and smoothing the fan
hypersurface, we obtain rational curves of degree $e\leq n$ on a general hypersurface of degree
$d\leq n$ in $\P^n$. However, we can do somewhat  better:
%We can now slightly extend the result of Lemma \ref{balanced-prop} to higher-degree curves:
\begin{thm}\label{balanced-high}
For $d\leq n$ and  $e\leq 2n-2$, there exists a rational curve of degree $e$ on a general hypersurface $X$
of degree $d$ in $\P^n$ having balanced normal bundle.
\end{thm}
\begin{proof}
Note that the case $e\leq n$ is handled 
 by Lemma \ref{secant-balanced-lem}
 (alternatively, by the case $C_2=\emptyset, a=e_1(d-1)$ of the
 constructions below); in some cases 
 Proposition \ref{balanced-prop} applies also, as well as
 the results of \cite{coskun-riedl}.  
 So we may assume $e>n$. 
 We use an  induction on $d$ based on the following 2 
constructions of curves
on a general hypersurface $X_1\cup X_2$ of bidegree $(d, d-1)$ on a 2-fan.\par
\ul{Construction 1:} $C_2\subset X_2$ is a general rational curve of degree $e_2= n-2$.
$R_1,...,R_{e_2}\subset \tilde Y\subset X_1$ are rulings connected to $C_2$,
i.e. $(\bigcup R_i)\cap E=C_2\cap E$. $C_1\subset X_1$
is a general curve of degree $e_1\leq n$  meeting $R_1$ and no other $R_i$, and meeting $\tilde Y$ in
$a-1=e_1(d-1)-1$ additional points (hence $C_1$ is disjoint from $E$). 
We let $C=C_1\cup\bigcup\limits_{i=1}^{e_2}R_i\cup C_2$. Note that this has dgree $e_1+e_2$
with respect to the line bundle $\O(1,0)$.
\par
\ul{Construction 2:} $C_2, R_2,...,R_{e_2}$ as above, $C_1$ meeting $\tilde Y$ in $e_1(d-1)$ points not
on any $R_i$, and meeting $C_2\cap E$ in 1 point.
We let $C=C_1\cup\bigcup\limits_{i=2}^{e_2}R_i\cup C_2$. It has $\O(1,0)$-degree equal to $e_1+e_2-1$.
\par
Regarding Construction 1, given $C_2\subset X_2$, a general degree-$e_2$ rational curve on a general
degree-$(d-1)$ hypersurface, we choose $f_d$ generally through $C_2\cap E$, which then determines 
the quasi-cone $Q(f_d, d_{d-1})$. Note that the normal bundle to the rulings is of the form
$(n-2)\O\oplus\O(-1)$ where the first summand is the normal bundle in $\tilde Y$. It follows easily
that
\[N_{C/X}|_{C_2}=N_{C_2/X_2}, N_{C/X}|_{R_i}=(n-2)\O\oplus \O(-1), i\geq 2,
 N_{C/X}|_{R_1}=(n-1)\O.\]
 Now the degree of $N_{C_2/X_2}$ is $e_2(n+2-d)-2=(n-2)(n+2-d)-2$
 	and its rank is $n-2$. Therefore is has upper rank $r_+=n-4$
 	(or $n-2=2$ if $n=4$).
Moreover $N_{C/X}|_{C_1}$ is an elementary 'up' modification 
(length-1 enlargement) of the secant
sheaf $N^s_{C_1}$ in the point $C_1\cap R_1$ (and thus locally equal to $N_{\bar C_1/\P^{n-1}}$
there where $\bar C_1$ denotes the image of $C_1$ in $\P^{n-1}$.  By general choice of $f_d$, the
tangent space involved can be chosen generally and consequently by 
 Lemma \ref{comb-lem}
 and induction, $N_{C/X}$
is balanced, hence likewise for the  general smoothing $(C', X')$ of $(C, X)$ where $C'$ has
degree $e_1+e_2$. The case of Construction 2
is similar. \par
Now for $d=1$, it is well known that there exist rational curves of any degree with balanced normal bundle.
Then using either of the above constructions plus induction and smoothing,
we construct rational curves of all degrees $e\leq 2n$  on $X$ of any degree $d\leq n$ with balanced normal bundles.
\end{proof}
\begin{rem}
The foregoing constructions require
the existence of a hypersurface $f_d$ in $\P^{n-1}$ with given tangent hyperplanes at $e_2$ points.
Counting dimensions, the latter 'should' be
 possible when $(n-1)e_2\leq \binom{n+d-1}{n-1}$. Thus,
this construction could in principle produce 
curves with balanced normal bundle and degree $e$
up to, roughly, $\frac{1}{n-1}\binom{n+d-1}{n-1}$ on a 
general hypersurface of degree $d$ in $\P^n$. However, I have no reason
to believe this would be sharp or for that matter that there is any
upper bound on the degrees of rational curves with balanced normal bundle
on a general Fano hypersurface.
\end{rem}
%\begin{prop}
%For a general pair $(F_{d-1}, F_d)$ and a general point $p\in\P^n$, the locus of $a$-secants to $Y=V(F_{d-1}, F_d)$
%through $p$ is irreducible and $(n-a-1)$-dimensional.
%\end{prop}
%\begin{proof}
%To begin with, the set of all tuples $(L, A, F_{d-1}, F_d)$ such that $L$ is a line through $p$,
%$A$ is a divisor of degree $a$ on $L$ contained in $Y$ is irreducible and smooth of the appropriate
%dimension, namely $n-1+N_{d-1}+N_d-a$.
%\end{proof}
%\end{proof}
\section{Irreducible families}
Here we use out techniques to
re- prove a result on irreducibility of families of rational curves
 degree $e<d$ on hypersurfaces
of degree $d<n$ in $\P^n$. The idea is to specialize to a 2-fan $X_1\cup X_2$ and
 try to get the
curve to specialize to $X_1$ by imposing the maximum number of points to go into $X_1$.
This 'maximum number' is known thanks to balancedness
of the normal bundle (cf. Cor. \ref{curve-thru-point-cor}). 
Then the part of the curve in $X_1$ is a suitable multisecant curve amenable to the
results of the last section. 
%The main difficulty with this approach is that the part in $X_1$ may a priori
%be disconnected, hence admit too many parameters. We deal with this by showing that the 'connecting
%points', i.e. 
%the points of the curve on $E=X_1\cap X_2$ are general on $E$
\begin{thm}\label{main}
The family of rational curves of degree $e<d$ on a general hypersurface  of degree $d$
in $\P^n$ for $d<n, n\geq 4$ is irreducible,
 generically unobstructed and of dimension
$e(n+1-d)+n-4$.
\end{thm}
\begin{rem}\label{referee}For $d\leq n-2$ and $e$ arbitrary, 
this has been proven by Riedl-Yang \cite{riedl-yang} based on bend-and-break. 
For $d=n-1$ some results for low $e$ were obtained by Tseng \cite{tseng-note}.
In fact, the referee notes that the Theorem also follows fully from the older work of
Harris-Roth-Starr \cite{harris-roth-starr}. The following
argument is due to him: \cite{harris-roth-starr} shows in the
given range that the locus of $e$-fold covered lines is irredcuible
and contained in a unique component of the closure of the locus of rational
curves; this plus monodromy-interchangeability (see below) imply that
there can be only one such component, QED.\end{rem}
\begin{proof}
To begin with, it follows from the result of \cite{glp} that the components
of the family of rational curves of degree $e$ on a general hypersurface $X$
of degree $d>e$ are monodromy-interchangeable. Let $S$ be one
of those components. 
By Lemma \ref{balanced-prop}, the general curve $C'$ in $S$ has balanced normal bundle
and hence by Corollary \ref{curve-thru-point-cor}, notations being 
as in \eqref{balanced-normal}, may be assumed to pass through 
a collection $(q'_\bullet)$ of  $a_++1$ general points while filling up an $r_++1$-dimensional
variety for each fixed $(q'_\bullet)$, provided $r_+<n-2$. In the exceptional case $r_+=n-2=\rk(N_{C/X})$,
we take $a_++2$ points instead, leading to a 1-dimensional variety, i.e. the curve $C'$ itself.\par 

Consider a degeneration of $X$ to a general hypersurface $X_1\cup X_2$ of bidegree $(d, d-1)$ 
on a 2-fan and with it a degeneration of $q'_\bullet$ to general points $(q_\bullet)$ on $X_1$ (only!).
Here $X_1$ is the blowup of a general quasi-cone of degree $d$
 at its vertex, 
$X_2$ is a general degree-$(d-1)$ hypersurface and $E=X_1\cap X_2$
is the exceptional divisor on $X_1$ and
 a hyperplane section of $X_2$, hence a degree-$(d-1)$ hypersurface in $\P^{n-1}$. 
As we have seen, $X_1$ can also be represented as the blowup of a general
$(d, d-1)$ complete intersection $Y=F_{d-1}\cap F_d$ in $\P^{n-1}$ with exceptional divisor $\tilde Y$ and 
$E$ being the proper transform of $F_{d-1}$.\par
Let $(C, f)$ be the Kontsevich 
 limit on $X_1\cup X_2$ of a general curve in $S$, considered as a 
stable (unpointed) map. This is the special fibre of a 1-parameter family
\eqspl{}{
\begin{matrix}
\mathcal C&\ \ &\to&\ \  &\mathcal X\\
&\searrow&&\swarrow&\\
&&B&&
\end{matrix}
}
As we have seen in Prop.\ref{no-E-prop}, no component of $C$ can map non-constantly to 
$E=X_1\cap X_2$. 
Let $C_i$ be the sum of the components of $C$ mapping non-constantly to $X_i$ (but not to $E$),
$i=1,2$ and let $m$ be the degree of $f_*(C_2)$, i.e. 
\[m=f_*(C_2).E=C_2.f^*(X_1)=-C_2.f^*(X_2).\]
Likewise,
\[m=f_*(C_1).E=C_1.f^*(X_2)=-C_1.f^*(X_1)\]
(note that $f^*(X_1), f^*(X_2)$ have degree 0 on any fibre component
mapping to a point).We call $(e, m)$ the \emph{type} of $f_*(C_1)$ (or, abusively, of $C_1$).
 Then $f_*(C_1)$ has $H$-degree $e-m$ and intersection number
$a=(e-m)(d-1)-m$ with $\tilde Y$ and
by upper semicontinuity,  the family of $C_1$s going through $(q_\bullet)$
fills up a variety that has dimension $r_++1$ or more- though possibly singular or non-reduced-
 locally at each $q_i$ ,
%is of dimension at least \[e(n-d+1)+n-4-(n-2)=e(n-d+1)-2>0,\]
 because this is true for the curves in $X$
through $(q_\bullet)$ (see Corollary \ref{curve-thru-point-cor}). \par

Now I claim that $C_1$ is irreducible except for
contracted components. 
Let $C_{1, 1}, ..., C_{1,k}$ be the connected components of $C_1$, let $(e_i, m_i)$ be the type of $C_{1, i}$,
and let $b_i$ be the number of points $q_j$ on $C_{1, i}$. Thus,
\[\sum b_i=a_++1=\lfloor\frac{e(n-d+1)-2}{n-2}\rfloor, e(n-d+1)-2=a_+(n-2)+r_+.\]
By Proposition \ref{no-E-prop}, no $C_{1, i}$ can map nonconstantly to 
$E$ and by Lemma \ref{inf-near-sec-lem}, we have $m_i\geq 0, \forall i$
(even if $C_{1,i}$ maps to $\tilde Y$).

Note that the degree of the secant bundle $N^s_{C_{1,i}}=N_{C_{1,i}/X}$
is $c_i=e_i(n-d+1)+m_i-2$ which we write as above in the form
\[a_{+i}(n-2)+r_{+i}, \ \ 0<r_{i+}\leq n-2,\]
Thus if $N^s_{C_{1,i}}$ is balanced then it has the form $r_{+i}\O(a_{+i}+1)\oplus 
(n-2-r_{+i})\O(a_{+i})$ but in any case 
any family of curves of type $(e_i, m_i)$ can go through at most $a_{+i}+1$
general points. 
Note that $c_i>0$: indeed the only case where $c_i=0$
is where $e_i=1, d=n-1, m=0$ and this would violate connectedness of $C$.
Also,
\[\sum c_i\leq e(n-d+1)-2.\]
\par

Now we assume $r_+<n-2$ as the case $r_+=n-2$ is similar and simpler.
Then since $\sum b_i=a_++1$, we have
\[(n-2)\sum b_i>e(n-d+1)-2\]
Therefore we may assume $b_1\geq a_{+1}+1$. But since $C_{1,1}$
goes through $b_1$ general points this forces $b_1=a_{+1}+1$.
Now suppose $k>1$. If $b_1=a_+$ (i.e. all the $q$ points go to $C_{1,1}$),
then we get $r_{+1}<r_+$ which is a contradiction
since as noted above $r_++1$ is the local dimension of the 
variety filled up by the curves going through the $q$ points.
Therefore we have $b_1<a_+$ so we may assume $b_2>0$.
But then freeing up a $q$ point on $C_{1,2}$ has no (dimension-raising) effect on the variety
locally filled up by the curves at some $q$ point on $C_{1,1}$, unlike the situation
on the general fibre, which is a contradiction.
\par
Now a similar dimension argument shows that we must have $m=0$. 
Then by lemma \ref{free-rational-lem}, $f$ is an embedding if $n\geq 4$.
Then we can invoke Theorem \ref{secant-thm} to conclude that the limit family is unique.
Since the limit occurs with multiplcity 1, it follows that
the family of curves of degree $e$ on $X$ is irreducible.
\end{proof}
\begin{rem}[speculative]
In the range $e\geq d$ not covered by our argument, The dimension of any component
$S$ of the family of rational curves is at least $d(n+1-d)+n-4$, which is $>2(n-1)$
provided $d$ is at least, roughly, $3n/4$ (the most interesting part
of the $d$ range is near $n-1$). Thus in this range bend-and-break is
applicable to conclude that $S$ contains some reducible curves. Then, some of the simpler
arguments from \cite{riedl-yang}, avoiding their
'borrowing' argument, might conceivably be used to conclude irreducibility,
but we have no proof.
\end{rem}
\section{Curves through a general point}\label{through-pt}
Let us denote by $R_e(X)$ (resp. $R_e(X, p)$) the set of rational curves of degree $e$
in $X$ (resp.  in $X$ and through the point $p\in X$). Here we show that under some circumstances
it is possible to deduce irreducibility of $R_e(X, p)$ for general $p$ from that of 
$R_e(X)$. As the referee notes, these results also follow from \cite{harris-roth-starr}
as in Remark \ref{referee}. 
\begin{prop}
Notations as above, assume $n\geq 4$ and
 $d\leq n-2$. Then $R_e(X, p)$ is irreducible
for general $p$.
\end{prop}
\begin{proof}By the theorem of Riedl-Yang
\cite{riedl-yang}, $R_e(X)$ is irreducible, hence 
 the components of $R_e(X, p)$ for general $p$ are monodromy-interchangeable
as $p$ ranges over $X$. Now we degenerate $X$ to a general $d$-fan
of type $(d, d-1,...,1)$, which is
 of the form $X(0)=X_d\cup X_{d-1}\cup...\cup
X_1$ where each $X_k$ is a blown up $k$-quasi-cone. Pick a general point $p\in X_1$
and consider the family of $e$-gons in $X_1$ of the form
$C=L_1\cup ...\cup L_e$, where $p\in L_1$, where the $L_i$ are lines and $L_i$ meets  $L_{i+1}, i<e$.
We may consider $Y=X_1\cap X_2\cap...\cap X_d\cap E$ as a general  $(d, d-1,...,1)$
complete intersection in $E=\P^{n-1}$. When $d\leq n-2$, $Y$ is irreducible.
Choosing $L_i$ so that $L_i\cap E\in Y$, $C$ may be extended 'trivially'
by attaching rulings, to a smoothable curve in $X(0)$. This leads to an irreducible
unique family of $e$-gons contained in a unique component of $F_e(X, p)$. By the monodromy-
interchangeability, it follows that $R_e(X, p)$ itself is irreducible.
\end{proof}
The argument is partly extendable to the case $d=n-1$:
\begin{prop}
Notations as above, assume $d=n-1\geq 3$ and $e<d$ . 
Then $R_e(X, p)$ is irreducible for general $p$.
\end{prop}
\begin{proof}
We argue similarly, now using Theorem \ref{main}, in lieu of \cite{riedl-yang} initially
to conclude that the components of $R_e(X, p)$ are monodromy-interchangeable 
(this is the only place where the assumption $e<d$ is used).
Only now we degenerate $X$ to $X_d\cup...\cup X_3\cup X'_2$ where $X'_2$ is 
a quadric. We will prove irreducibility for $e=2, 3$ using conics and twisted cubics
in $X'_2$ extended 'trivially' via rulings in $\tilde Y$. Then the general case follows
similarly by using chains of conics or chains consisting of conics plus
one twisted cubic in $X'_2$ extended via rulings.
Note that in this case $Y=E\cap X_d\cap...\cap X'_2$ is an irreducible curve. 
The conics in question are of the form $M\cap X'_2$ where $M$ is a $\P^2$
spanned by $q_1, q_2\in Y$ and $p\in X'_2\setminus E$ and the set of such, for fixed $p$,
is obviously irreducible. This settles the case $e=2$.\par
The case of twisted cubics, i.e. $e=3$,
 is harder because for fixed, general $q_1, q_2, q_3\in Y$
and $M=\carets {q_1, q_2, q_3, p}$, $M\cap X'_2$ is a smooth quadric surface which carries
\emph{two} 1-parameter families of twisted cubics through the 4 points, each complementary
or 'linked' to a family of lines (this already shows that $R_3(X, p)$
has at most 2 components and then using chains with one cubic and some conics
it follows that the same is true for $R_e(X, p)$ for all $e\geq 3$ odd). 
We are however claiming that $R_3(X, p)$ and then $R_e(X, p)$ has
just one component.
To this end, specialize $q_\bullet$ to $q_\bullet^0$ where $q_\bullet^0\in T\cap Y$
where $T$ is the tangent hyperplane to $X'_2$ at $p$. In this special case 
$M^0\cap X'_2=\carets{q^0_\bullet}\cap X'_2$
is a quadric cone with vertex $p$, which carries a unique family of twisted cubics $C$
through $p, q^0_\bullet$ linked to lines through the vertex $p$. 
If the locus of twisted cubics in $X'_2$ trisecant to $Y$ is reducible, it would be
singular at $[C]$. The Zariski tangent space ot this locus is $H^0$ of 
the secant bundle $N^s_C(-p)$,
where $N_C$ is the normal bundle in $X'_2$, which has degree $3n-5$, and $N^s_C$ is
the elementary modification of $N_C$ corresponding to the tangent spaces
$T_{q^0_i}Y, i=1,2,3$.\par
Now it is easy to check that for a twisted cubic $C$ contained is a smooth quadric surface $Q$
inside $X'_2$ the normal bundle $N_{C/X'_2}$
is $\O(4, 3^{n-3})$ with $\O(4)=N_{C/Q}$, $\O(3^{n-3})=N_{Q, X'_2}|_C$,
while for $C$ in a quadric cone surface the normal bundle becomes $\O(5, 3^{n-4}, 2)$
(which shows that $C$ is a singular point in the space of twisted cubics through a \emph{fixed}
quadruple $(q^0_\bullet, p)$, because twisting down by these points yields a (-2)
quotient).
Now because the tangent $T_{q^0_i}$ can be chosen generally, the secant bundle $N^s_C$
is a general elementary modification of $N_{C/X'_2}$ of colength $n-2$
locally at each of the  $q^0_i, i=1,2,3$ 
so $N^s_C$ is of type $\O(2,1,1,0^{n-5})$, $N^s_C(-p)$ is of type
$\O(1, 0,0, (-1)^{n-5})$, and consequently
$H^1(N^s_C(-p))=0$, which makes $C$ a nonsingular point in the space of
twisted cubics trough $p$ trisecant to $Y$ (allowing $q^0_\bullet$ to move on $Y$), contradiction.
\end{proof}

\bibliographystyle{amsplain}
\bibliography{../mybib}
\end{document}